\documentclass[10pt,reqno,oneside]{amsart}
\usepackage{graphicx,verbatim}
\usepackage{amssymb,cite}
\usepackage{epstopdf}
\usepackage{graphicx}
\usepackage{amsmath}
\usepackage{amsthm, enumerate}
\usepackage{mathrsfs}
\usepackage{caption}
\usepackage{subcaption}
\usepackage[normalem]{ulem}   
\allowdisplaybreaks

\chardef\forshowkeys=0
\chardef\showllabel=0
\chardef\refcheck=0
\chardef\sketches=0
\chardef\showfont=0         

\usepackage{float}

\floatstyle{boxed}

\usepackage{marginnote}

\usepackage[colorlinks=true, pdfstartview=FitV, linkcolor=black, citecolor=black, urlcolor=black]{hyperref}

\ifnum\forshowkeys=1
  
  \usepackage[notref,notcite,color]{showkeys}
\fi

\author[M.S.~Ayd{\i}n]{Mustafa Sencer Ayd{\i}n}
\address{Department of Mathematics, University of Southern California, Los Angeles, CA 90089}
\email{maydin@usc.edu}

\title{Euler Equations in Sobolev conormal spaces}
\author[I.~Kukavica]{Igor Kukavica}
\address{Department of Mathematics\\
 University of Southern California\\
 Los Angeles, CA 90089}
\email{kukavica@usc.edu}

\usepackage{enumitem}
\usepackage{datetime}
\usepackage{fancyhdr}
\usepackage{comment}
\allowdisplaybreaks
\usepackage[margin=1in]{geometry}
\usepackage{amsmath, amsthm, amssymb}
\usepackage{times}
\usepackage{graphicx}
\usepackage[usenames,dvipsnames,svgnames,table]{xcolor}

\ifnum\showfont=1
  \usepackage{fontspec}
\fi

\ifnum\refcheck=1
\usepackage{refcheck}
\fi

\begin{document}
 
\def\inprogress{{\colg IN PROGRESS} }
\def\bnew{\colr {}}
\def\enew{\colb {}}
\def\bold{\colu {}}
\def\eold{\colb{}}
\def\YY{X}
\def\OO{\mathcal O}
\def\SS{\mathbb S}
\def\CC{\mathbb C}
\def\RR{\mathbb R}
\def\TT{\mathbb T}
\def\ZZ{\mathbb Z}
\def\HH{\mathbb H}
\def\RSZ{\mathcal R}
\def\LL{\mathcal L}
\def\SL{\LL^1}
\def\ZL{\LL^\infty}
\def\GG{\mathcal G}
\def\tt{\langle t\rangle}
\def\erf{\mathrm{Erf}}
\def\mgt#1{\textcolor{magenta}{#1}}
\def\ff{\rho}
\def\gg{G}
\def\sqrtnu{\sqrt{\nu}}
\def\ww{w}
\def\ft#1{#1_\xi}
\def\ges{\gtrsim}
\renewcommand*{\Re}{\ensuremath{\mathrm{{\mathbb R}e\,}}}
\renewcommand*{\Im}{\ensuremath{\mathrm{{\mathbb I}m\,}}}

\ifnum\showllabel=1
\def\llabel#1{\marginnote{\color{gray}\rm(#1)}[-0.0cm]\notag}
\else
\def\llabel#1{\notag}
\fi

\newcommand{\norm}[1]{\left\|#1\right\|}
\newcommand{\nnorm}[1]{\lVert #1\rVert}
\newcommand{\abs}[1]{\left|#1\right|}
\newcommand{\NORM}[1]{|\!|\!| #1|\!|\!|}
\theoremstyle{plain}

\newtheorem{Theorem}{Theorem}[section]
\newtheorem{Proposition}[Theorem]{Proposition}
\newtheorem{Lemma}[Theorem]{Lemma}
\newtheorem{Corollary}[Theorem]{Corollary}
\newtheorem{Assumption}[Theorem]{Assumption}
\newtheorem{Definition}[Theorem]{Definition}
\theoremstyle{definition}
\newtheorem{Remark}[Theorem]{Remark}
\newtheorem{Notation}[Theorem]{Notation}
\newtheorem{Example}[Theorem]{Example}
\newtheorem{Exercise}[Theorem]{Exercise}
\def\theequation{\thesection.\arabic{equation}}
\numberwithin{equation}{section}
\definecolor{mygray}{rgb}{.6,.6,.6}
\definecolor{myblue}{rgb}{9, 0, 1}
\definecolor{colorforkeys}{rgb}{1.0,0.0,0.0}
\newlength\mytemplen
\newsavebox\mytempbox
\makeatletter
\newcommand\mybluebox{%
  \@ifnextchar[
  {\@mybluebox}%
  {\@mybluebox[0pt]}}
\def\@mybluebox[#1]{%
  \@ifnextchar[
  {\@@mybluebox[#1]}%
  {\@@mybluebox[#1][0pt]}}
\def\@@mybluebox[#1][#2]#3{
  \sbox\mytempbox{#3}%
  \mytemplen\ht\mytempbox
  \advance\mytemplen #1\relax
  \ht\mytempbox\mytemplen
  \mytemplen\dp\mytempbox
  \advance\mytemplen #2\relax
  \dp\mytempbox\mytemplen
  \colorbox{myblue}{\hspace{1em}\usebox{\mytempbox}\hspace{1em}}}
 \makeatother
\def\aand{{\indeq}\text{and}{\indeq}}
\def\XX{{\mathcal X}}
\def\XXT{{\mathcal X}_T}
\def\XXTzero{{\mathcal X}_{T_0}}
\def\XXi{{\mathcal X}_\infty}
\def\YY{{\mathcal Y}}
\def\YYT{{\mathcal Y}_T}
\def\YYTzero{{\mathcal Y}_{T_0}}
\def\YYi{{\mathcal Y}_\infty}
\def\cc{\text{c}}
\def\rr{r}
\def\weaks{\text{\,\,\,\,\,\,weakly-* in }}
\def\inn{\text{\,\,\,\,\,\,in }}
\def\cof{\mathop{\rm cof\,}\nolimits}
\def\Dn{\frac{\partial}{\partial N}}
\def\Dnn#1{\frac{\partial #1}{\partial N}}
\def\tdb{\tilde{b}}
\def\tda{b}
\def\qqq{u}
\def\lat{\Delta_2}
\def\biglinem{\vskip0.5truecm\par==========================\par\vskip0.5truecm}
\def\inon#1{\hbox{\ \ \ \ \ \ \ }\hbox{#1}}                
\def\onon#1{\inon{on~$#1$}}
\def\inin#1{\inon{in~$#1$}}
\def\FF{F}
\def\andand{\text{\indeq and\indeq}}
\def\ww{w(y)}
\def\ll{{\color{red}\ell}}
\def\ee{\epsilon_0}
\def\startnewsection#1#2{ \section{#1}\label{#2}\setcounter{equation}{0}}   
\def\loc{\text{loc}}
\def\nnewpage{ }
\def\sgn{\mathop{\rm sgn\,}\nolimits}    
\def\Tr{\mathop{\rm Tr}\nolimits}    
\def\div{\mathop{\rm div}\nolimits}
\def\curl{\mathop{\rm curl}\nolimits}
\def\dist{\mathop{\rm dist}\nolimits}  
\def\supp{\mathop{\rm supp}\nolimits}
\def\indeq{\quad{}}           
\def\period{.}                       
\def\semicolon{\,;}                  
\def\colr{\color{red}}
\def\colrr{\color{black}}
\def\colb{\color{black}}
\def\coly{\color{lightgray}}
\definecolor{colorgggg}{rgb}{0.1,0.5,0.3}
\definecolor{colorllll}{rgb}{0.0,0.7,0.0}
\definecolor{colorhhhh}{rgb}{0.3,0.75,0.4}
\definecolor{colorpppp}{rgb}{0.7,0.0,0.2}
\definecolor{coloroooo}{rgb}{0.45,0.0,0.0}
\definecolor{colorqqqq}{rgb}{0.1,0.7,0}
\def\colg{\color{colorgggg}}
\def\collg{\color{colorllll}}
\def\coleo{\color{colorpppp}}
\def\cole{\color{black}}
\def\colu{\color{blue}}
\def\colc{\color{colorhhhh}}
\def\colW{\colb}   
\definecolor{coloraaaa}{rgb}{0.6,0.6,0.6}
\def\colw{\color{coloraaaa}}
\def\comma{ {\rm ,\qquad{}} }            
\def\commaone{ {\rm ,\quad{}} }          
\def\lec{\lesssim}
\def\nts#1{{\color{red}\hbox{\bf ~#1~}}} 
\def\ntsf#1{\footnote{\color{colorgggg}\hbox{#1}}}
\def\ntsik#1{{\color{purple}\hbox{\bf ~#1~}}} 
\def\blackdot{{\color{red}{\hskip-.0truecm\rule[-1mm]{4mm}{4mm}\hskip.2truecm}}\hskip-.3truecm}
\def\bluedot{{\color{blue}{\hskip-.0truecm\rule[-1mm]{4mm}{4mm}\hskip.2truecm}}\hskip-.3truecm}
\def\purpledot{{\color{colorpppp}{\hskip-.0truecm\rule[-1mm]{4mm}{4mm}\hskip.2truecm}}\hskip-.3truecm}
\def\greendot{{\color{colorgggg}{\hskip-.0truecm\rule[-1mm]{4mm}{4mm}\hskip.2truecm}}\hskip-.3truecm}
\def\cyandot{{\color{cyan}{\hskip-.0truecm\rule[-1mm]{4mm}{4mm}\hskip.2truecm}}\hskip-.3truecm}
\def\reddot{{\color{red}{\hskip-.0truecm\rule[-1mm]{4mm}{4mm}\hskip.2truecm}}\hskip-.3truecm}
\def\gdot{\greendot}
\def\tdot{\gdot}
\def\bdot{\bluedot}
\def\ydot{\cyandot}
\def\rdot{\reddot}
\def\fractext#1#2{{#1}/{#2}}
\def\ii{\hat\imath}
\def\fei#1{\textcolor{blue}{#1}}
\def\vlad#1{\textcolor{cyan}{#1}}
\def\igor#1{\text{{\textcolor{colorqqqq}{#1}}}}
\def\igorf#1{\footnote{\text{{\textcolor{colorqqqq}{#1}}}}}
\newcommand{\p}{\partial}
\newcommand{\UE}{U^{\rm E}}
\newcommand{\PE}{P^{\rm E}}
\newcommand{\KP}{K_{\rm P}}
\newcommand{\uNS}{u^{\rm NS}}
\newcommand{\vNS}{v^{\rm NS}}
\newcommand{\pNS}{p^{\rm NS}}
\newcommand{\omegaNS}{\omega^{\rm NS}}
\newcommand{\uE}{u^{\rm E}}
\newcommand{\vE}{v^{\rm E}}
\newcommand{\pE}{p^{\rm E}}
\newcommand{\omegaE}{\omega^{\rm E}}
\newcommand{\ua}{u_{\rm   a}}
\newcommand{\va}{v_{\rm   a}}
\newcommand{\omegaa}{\omega_{\rm   a}}
\newcommand{\ue}{u_{\rm   e}}
\newcommand{\ve}{v_{\rm   e}}
\newcommand{\omegae}{\omega_{\rm e}}
\newcommand{\omegaeic}{\omega_{{\rm e}0}}
\newcommand{\ueic}{u_{{\rm   e}0}}
\newcommand{\veic}{v_{{\rm   e}0}}
\newcommand{\up}{u^{\rm P}}
\newcommand{\vp}{v^{\rm P}}
\newcommand{\tup}{{\tilde u}^{\rm P}}
\newcommand{\bvp}{{\bar v}^{\rm P}}
\newcommand{\omegap}{\omega^{\rm P}}
\newcommand{\tomegap}{\tilde \omega^{\rm P}}
\renewcommand{\up}{u^{\rm P}}
\renewcommand{\vp}{v^{\rm P}}
\renewcommand{\omegap}{\Omega^{\rm P}}
\renewcommand{\tomegap}{\omega^{\rm P}}
\def\hh{\text{h}}
\def\cco{\text{co}}

\begin{abstract}
  We consider the three-dimensional incompressible Euler equations in Sobolev conormal spaces and establish local-in-time existence and uniqueness in the half-space or channel. The initial data is Lipschitz having four square-integrable conormal derivatives and two bounded conormal derivatives. We do not impose any integrability or differentiability assumption for the normal derivative.
\end{abstract}

\maketitle

\startnewsection{Introduction}{sec.int}

We address the local-in-time existence and uniqueness, in Sobolev conormal spaces, of the three-dimensional incompressible Euler equations
 \begin{align}
  u_t + u\cdot \nabla u + \nabla p =0
  \comma \nabla \cdot u = 0
  \comma (x,t) \in \Omega \times (0,T)
  ,
  \llabel{euler}
 \end{align}
with the slip boundary condition
  \begin{align}
  u \cdot n = 0 
  \text{ on } \partial \Omega \times (0,T)
  \llabel{eulerb}
 \end{align}
where $n$ is the unit outward normal. 

One of the earliest results on local well-posedness is due to Lichtenstein in~\cite{L}, who considered the data in $C^{k,\alpha}$.  Kato, in~\cite{K1} and~\cite{K2}, addressed the problem in Sobolev spaces $H^m(\mathbb{R}^3)$, for $m\ge 3$ and $H^s(\mathbb{R}^3)$, for $s > \frac{5}{2}$, respectively.  Later, in~\cite{KP}, Kato and Ponce extended this result by considering the initial data in $W^{s,p}(\mathbb{R}^d)$, for $s > \frac{d}{p}+1$.  See~\cite {C1,C2,C3,CW,GL,GLY,PP} for other approaches in different functional settings.
  
A substantial amount of work also considered the ill-posedness of the three-dimensional incompressible Euler equations, which may refer to either the failure of existence, uniqueness, or the continuous dependence of the solution map on the initial data. In terms of the non-existence results, one of the earlier works is due to DiPerna and Lions, in~\cite{DL}, who considered initial data in $W^{1,p}$, for $p<\infty$.  Next, in \cite{BT}, Bardos and Titi, obtained a similar result for $u_0 \in C^\alpha$ where $\alpha <1$ and for the critical level Besov spaces, see~\cite{MY1}.  Subsequently, in~\cite{BL1}, Bourgain and Li studied this problem in $C^m$, proving non-existence. Simultaneously and employing a different approach, in~\cite{EM}, Elgindi and Masmoudi established a similar result in $C^k \cap L^2$. Then, Bourgain and Li, in~\cite{BL2}, extended the non-existence result to the critical level Sobolev spaces.  We note that these results provide constructions of initial data for which the Euler equations do not admit a solution keeping the same level of regularity.  For other ill-posedness results, such as non-uniqueness or discontinuous dependence on the initial data, see, for example,~\cite{BDIS, DS, HM, I, MY2, Sc,Sh}.
 
The previously mentioned results are concerned with
spatially isotropic functional spaces, i.e.,
derivatives in all directions are assumed to have the same level of regularity.
This assumption is natural when 
the domain under consideration, such as $\mathbb{R}^3$ or $\mathbb{T}^3$,
does not have a boundary. 
However, the presence of the boundary allows for a normal direction, and derivatives 
in this direction may behave differently in the vicinity of the boundary
than the tangential derivatives. From a mathematical point of view,
this asymmetry regarding the behavior of the derivatives is more
apparent when we consider the inviscid limit problem. 
The strong vanishing viscosity limit 
has been established by numerous works when the Navier-Stokes equations
  \begin{align}
   \partial_t u^\nu
     - \nu \Delta u^\nu
     + u^\nu \cdot \nabla u^\nu
     + \nabla p^\nu
     =0 
   \comma
   \nabla \cdot u^\nu = 0 \comma (x,t) \in \Omega \times (0,T)
   ,
   \llabel{NSE}
  \end{align}
are coupled with the Navier boundary conditions
  \begin{align}
   u^\nu\cdot n = 0 \comma 
   \left(\frac{1}{2}(\nabla u^\nu + \nabla^T u^\nu)\cdot n\right)_\tau
   =-\mu u^\nu_\tau
   \onon{\partial\Omega}
   ,
   \llabel{navierbdry}
  \end{align}
or other related Navier-type slip conditions---see~\cite{BC1,BC2,BC3,BC4,BS1,BS2,CQ,DN,GK,IS,IP,NP,TWZ,WXZ,X,XX,XZ1,XZ2}.  Intuitively, Navier-type conditions imply that the first-order normal derivative of $u^\nu$ evaluated at the boundary is tangential. Therefore, the resulting boundary layer is weaker compared to the case where we impose no-slip condition $u^\nu|_{\partial \Omega} = 0$.  Therefore, one may expect different behaviors for the normal and tangential derivatives near the boundary.
 
Masmoudi and Rousset considered the inviscid limit problem in the functional framework of Sobolev conormal spaces. In~\cite{MR1}, the authors assumed that
  \begin{align}
   (u,\nabla u)|_{t=0}  \in H^7_\cco \times (H^6_\cco \cap W^{1,\infty}_\cco)
   ,\label{mrassumption}
  \end{align}
(see the definitions~\eqref{con.def} and~\eqref{norm} below) and established the vanishing viscosity limit.  This shows that the Euler equations are well-posed over a functional setting where only one normal derivative is allowed.  In~\cite{AK}, for the case of the half-space, we extended their result to a weaker class of initial data given by
  \begin{align}
   (u,\nabla u)|_{t=0}  \in (H^4_\cco \cap W^{2,\infty}_{\cco}) \times (H^2_\cco \cap L^\infty)
   ,\label{ourassumption}
  \end{align}
and
  \begin{align}
   (u,\nabla u)|_{t=0}  \in (H^4_\cco \cap W^{2,\infty}_{\cco}) \times (H^1_\cco \cap L^\infty)
   ,
   \llabel{ourassumption2}
  \end{align}
assuming additionally that $\mu \ge 0$; see~\cite[Theorems~2.1--2.3]{AK}.  In the same work, we announced that by a different construction method, it is possible to establish the existence and uniqueness of solutions of the Euler equations under the assumption
  \begin{align}
   (u,\nabla u)|_{t=0}  \in (H^4_\cco \cap W^{2,\infty}_{\cco}) \times L^\infty
  ,
  \label{ourassumption3}
  \end{align}
which is the main result of the present paper.  We now discuss the essential differences between the assumptions in~\eqref{mrassumption} and~\eqref{ourassumption3}.  First, we note that to propagate $\nabla u|_{t=0} \in W^{1,\infty}_\cco$, as in~\eqref{mrassumption}, at least six and five conormal derivatives on $u$ and $\nabla u$, respectively, are required; see~\cite{MR2}.  The need to have $\nabla u \in W^{1,\infty}_\cco$ uniformly in positive time stems from the hyperbolic nature of the Euler equations and a derivative loss.  For example, in~\cite{MR1}, the authors estimate $\nabla u \in H^{m-1}_\cco$ as
  \begin{align}
   \frac{d}{dt}\Vert \nabla u\Vert_{H^{6}_\cco}^2
   \lec
   \Vert \nabla p\Vert_{H^{6}_\cco}\Vert \nabla u\Vert_{H^{6}_\cco}
   +(1+ \Vert \nabla u\Vert_{W^{1,\infty}_\cco})
   (\Vert \nabla u\Vert_{H^{6}_\cco}^2+\Vert u\Vert_{H^{7}_\cco}^2)
   ,\llabel{mr1}
  \end{align}
assuming that $\nu =0$.  Removing $\Vert \nabla u^\nu\Vert_{W^{1,\infty}_\cco}$ and keeping six conormal derivatives on~$u^\nu$ at the same time is challenging.  In the context of the inviscid limit problem, we bypass this difficulty by introducing two conormal derivatives of $\nabla u$; see~\cite[Proposition~4.1]{AK}, while for the Euler equations, we only have the boundedness assumption on~$\nabla u$.  Having neither integrability nor differentiability assumptions on $\nabla u$ introduces new challenges in estimating $u$ in~$H^4_\cco$.  In~\cite{MR1}, the authors estimate this term as
  \begin{align}
  \frac{d}{dt}\Vert u\Vert_{H^{m}_\cco}^2
  \lec
  \Vert \nabla p\Vert_{H^{m-1}_\cco}\Vert u\Vert_{H^{m}_\cco}
  +(1+ \Vert \nabla u\Vert_{L^\infty})
  (\Vert \nabla u\Vert_{H^{m-1}_\cco}^2+\Vert u\Vert_{H^{m}_\cco}^2)
  ,
  \llabel{mr2}
  \end{align}
for all integer~$m \ge 0$, assuming that $\nu = 0$. 
In our setting, this estimate requires that $\nabla u^\nu \in H^3_\cco$,
uniformly,
which is unrealistic. Therefore, for the conormal derivatives of $u$, we establish  
  \begin{align}
  \frac{d}{dt}\Vert u\Vert_{4}^2
  \lec
  \Vert \nabla p\Vert_{H^{3}_\cco}\Vert u\Vert_{H^{4}_\cco}
  +
  (1+ \Vert \nabla u\Vert_{L^\infty}+ \Vert u\Vert_{W^{2,\infty}_\cco})
  \Vert u\Vert_{H^{4}_\cco}^2
  ,
  \llabel{our2}
 \end{align}
 eliminating the term $\Vert \nabla u^\nu\Vert_{H^m_\cco}$;
 see Proposition~\ref{P.Con}.
 This estimate hints at the independence of conormal derivatives of $u$
 from the conormal derivatives of~$\nabla u$.
 A similar comment applies to the pressure term; see Proposition~\ref{P.Pre}.
 Another difference between \eqref{mrassumption} and \eqref{ourassumption3}
 is the requirement $u_0 \in W^{2,\infty}_\cco$.
 In fact, it is possible to obtain from \eqref{EQ.emb} that
 \begin{align}
   \Vert u\Vert_{W^{2,\infty}_\cco}^2
   \lec
   \Vert \nabla u\Vert_{H^{2+s_1}_\cco}\Vert u\Vert_{H^{2+s_2}_\cco}
   ,\llabel{hypo}
 \end{align}
for $s_1+s_2 = 2+$.   
This means that $u \in W^{2,\infty}_\cco$ is a consequence of
$u\in H^5_\cco$ and~$\nabla u\in H^{2}_\cco$.
Since we have much weaker assumptions,   
we do not rely on any embedding 
to propagate the assertion $ u \in W^{2,\infty}_\cco$ uniformly on some time interval.
Therefore, we have to estimate this term directly.
Note that, in~\cite[Theorem3]{BILN}, the authors
obtained the existence for the Euler equations
if the initial data belongs to the space
  \begin{align}
   (u,\nabla u,\curl u)|_{t=0}
   \in
   H^4_\cco
     \times
   H^3_\cco
     \times
   W^{1,\infty}_\cco
   ,
   \llabel{EQ06}
  \end{align}
in the case of a general domain.
 
Theorem~\ref{T04} is a new existence result for the Euler equations.
In addition, 
it is an extension of the previously mentioned existence results 
since the required regularity level on the initial normal derivative of $u$ is minimal. 
We note in passing that 
it is challenging to improve on the a~priori estimates given in~Proposition~\ref{P.Ap} since
we cannot establish the short time persistence of
$u \in W^{2,\infty}_\cco$ when we work with fewer
integer conormal derivatives on either $u$ or~$\nabla u$.
Finally, we note that our results still hold when
in~\eqref{ourassumption3} we  
require $(\omega_1,\omega_2) \in L^\infty$
instead of $\nabla u_0 \in L^\infty$.

The paper is structured as follows. Section~\ref{sec.pre} contains the main result, the fundamental commutator lemma, and a summary of a~priori estimates.
The a~priori estimates are proven in detail in Section~\ref{sec.apri}, containing the conormal, pressure, and uniform velocity estimates.
The last section then contains the proof of the main existence theorem.

\startnewsection{Preliminaries and the Main Result}{sec.pre}
We consider $\Omega =\mathbb{R}^3_+$, and 
 for $x \in \Omega$ denote 
 $x = (x_\hh,z) = (x_1,x_2,z) \in \mathbb{R}^2 \times \mathbb{R}_+$.
We consider the Euler equations
 \begin{align}
  \partial_t u + u \cdot \nabla u + \nabla p =0 
  \comma
  \nabla \cdot u = 0 \comma (x,t) \in \Omega \times (0,T)
  ,
  \label{NSE0}
 \end{align}
subject to the slip boundary condition  
 \begin{align}
  u_3 = 0 \comma (x,t)\in \{z=0\}\times (0,T).
  \label{hnavierbdry}
 \end{align}
To state our main result, we first introduce conormal Sobolev  spaces.
Denote $\varphi(z) = z/(1+z)$, and let
$Z_1=\partial_1$, $Z_2=\partial_2$, and $Z_3=\varphi \partial_z$.
Then introduce
 \begin{align}
  \begin{split}
   H^m_\cco(\Omega)
   =&
   \{f \in L^2(\Omega) : Z^\alpha f \in L^2(\Omega),\  \alpha \in \mathbb{N}_0^3,\  0\le |\alpha|\le m  \},
   \\
   W^{m,\infty}_\cco(\Omega)
   =&
   \{f \in L^\infty(\Omega) : Z^\alpha f \in L^\infty(\Omega),\  \alpha \in \mathbb{N}_0^3,\  0\le |\alpha|\le m  \}
   ,
   \label{con.def}
  \end{split}
 \end{align} 
 equipped with the norms
 \begin{align}
  \begin{split}
   \Vert f\Vert_{H^m_\cco(\Omega)}^2
   =&
   \Vert f\Vert_{m}^2
   =
   \sum_{|\alpha|\le m} \Vert Z^\alpha f\Vert_{L^2(\Omega)}^2,
   \\
   \Vert f\Vert_{W^{m,\infty}_\cco(\Omega)}
   =&
   \Vert f\Vert_{m,\infty}
   =
   \sum_{|\alpha|\le m} \Vert Z^\alpha f\Vert_{L^\infty(\Omega)}
   .
   \label{norm}
  \end{split}
 \end{align}
 We use $\Vert f\Vert_{L^2}$ and $\Vert f\Vert_{L^\infty}$
 to denote the usual $L^2$ and $L^\infty$ norms, respectively.
 The following theorem is our main result.
 
\cole
\begin{Theorem}[Euler equations in the Sobolev conormal spaces]
\label{T04}
Let $u_0 \in H^4_\cco(\Omega) \cap W^{2,\infty}_\cco(\Omega) \cap W^{1,\infty}(\Omega)$
be such that $\div u_0 = 0$, and $u_0 \cdot n= 0$ on~$\partial \Omega$.
For some $T>0$,
there exists  a unique solution 
$u \in L^\infty(0,T;H^4_\cco(\Omega)\cap W^{2,\infty}_\cco(\Omega)\cap W^{1,\infty}(\Omega))$
to the Euler equations
\eqref{NSE0}--\eqref{hnavierbdry}
such that
   \begin{align}
    \sup_{[0,T]}
    (\Vert u(t)\Vert_{4}^2
    +\Vert u(t)\Vert_{2,\infty}^2
    +\Vert \nabla u(t)\Vert_{L^\infty}^2)
    \le M,
    \llabel{EQ.main3}
   \end{align}
for $M>0$ depending on the norms of the initial data. 
 \end{Theorem}
 \colb
 
 We note that $Z_1$ and $Z_2$
 commute with $\partial_z$, whereas $Z_3$ does not;
thus we frequently
separate $Z_\hh$ and $Z_{3}$ by writing
 \begin{align}
    Z^\alpha = Z^{\tilde{\alpha}}_\hh Z^k_3
    \comma \alpha = (\tilde{\alpha}, k) \in \mathbb{N}_0^2 \times \mathbb{N}_0
    .
    \llabel{ZhZ3}
 \end{align} 
To measure the commutator between $Z_3$ and $\partial_{z}$, we use the
 following identities.
 
 \cole
 \begin{Lemma}
  \label{L01}
Let $f$ be a smooth function. Then there exist smooth bounded functions
$\{c^k_{j,\varphi}\}_{j=0}^k$ and $\{\tilde{c}^k_{j,\varphi}\}_{j=0}^k$ of $z$, for $k \in \mathbb{N}$,
such that
  \begin{align}
    \begin{split}    
      &(i)\ \  Z^k_3 \partial_z f 
      =
      \sum_{j=0}^{k}
      c^k_{j,\varphi} \partial_z Z^j_3 f
      =\partial_z Z_3^kf + \sum_{j=0}^{k-1}
      c^k_{j,\varphi} \partial_z Z^j_3 f
      ,
      \\
      &(ii)\ \  \partial_z Z^k_3 f 
      =
      \sum_{j=0}^{k}
      \tilde{c}^k_{j,\varphi} Z^j_3 \partial_z f
      =Z_3^k \partial_z f +\sum_{j=0}^{k-1}
      \tilde{c}^k_{j,\varphi} Z^j_3 \partial_z f
      ,
      \\
      &(iii)\ \  Z^k_3 \partial_{zz} f 
      =
      \sum_{j=0}^{k}
      \sum_{l=0}^{j}
      \left(
      c^j_{l,\varphi} c^k_{j,\varphi} \partial_{zz} Z^l_3 f 
      +
      (c^j_{l,\varphi})' c^k_{j,\varphi} \partial_{z} Z^l_3 f 
      \right)
      ,
      \\
      &(iv)\ \  \partial_{zz} Z^k_3 f 
      =
      \sum_{j=0}^{k}
      \sum_{l=0}^{j}
      \tilde{c}^j_{l,\varphi} \tilde{c}^k_{j,\varphi} Z^l_3 \partial_{zz} f 
      +
      \sum_{j=0}^{k}
      (\tilde{c}^k_{j,\varphi})' Z^j_3 \partial_z f 
      ,
      \label{EQL02}
    \end{split}
  \end{align}
  where $\tilde{c}^k_{k,\varphi} = 1 = c^k_{k,\varphi}$.
 \end{Lemma}
 \colb

Note that the prime indicates the derivative with respect to the
variable~$z$. Also, note that the functions
${c}^k_{j,\varphi}$ and $\tilde{c}^k_{j,\varphi}$ depend on~$\phi$.

In addition to Lemma~\ref{L01}, we shall utilize the inequalities
 \begin{align}
  \Vert Z^{\alpha_1}f 
  Z^{\alpha_2}g\Vert_{L^2}
  \lec
  \Vert f\Vert_{L^\infty}
  \Vert g\Vert_{k}
  +
  \Vert f\Vert_{k}
  \Vert g\Vert_{L^\infty}
  \comma |\alpha_1|+|\alpha_2|=k\in\mathbb{N}_0       
  \comma f,g \in L^\infty \cap H^k_\cco
  \label{EQ.int}
 \end{align}
 and
 \begin{align}
  \Vert f\Vert_{L^\infty}^2
  \lec
  \Vert (I-\Delta_h)^\frac{s_1}{2}\partial_z f\Vert_{L^2}
  \Vert (I-\Delta_h)^\frac{s_2}{2}f\Vert_{L^2}
  \comma s_1,s_2 \ge 0 \comma s_1+s_2 >2.
  \label{EQ.emb}
 \end{align}
The proofs of \eqref{EQ.int} and  \eqref{EQ.emb}
can be found in~\cite{Gu} and~\cite{MR2}, respectively.
We shall also employ \eqref{hnavierbdry}, the Hardy inequality, and the 
incompressibility requirement to write
  \begin{align}
   \left\Vert \frac{u_3}{\varphi}\right\Vert_{W^{k,p}_\cco}
   \lec \Vert Z_\hh u_\hh\Vert_{W^{k,p}_\cco}
   \comma
   p \in [1,\infty],
   \label{EQ.u3}
  \end{align}
where $W^{k,p}_\cco$ is defined
analogously to~\eqref{con.def}.
Now, we state our a~priori estimates.
 
 \cole
 \begin{Proposition}[A priori estimates]
  \label{P.Ap}
  Given $u_0 \in H^4_\cco(\Omega) \cap W^{2,\infty}_\cco(\Omega) \cap W^{1,\infty}(\Omega)$,
  there exists
\begin{align}
   T_0 = T_0(\Vert u_0\Vert_{4},\Vert u_0\Vert_{2,\infty},\Vert \nabla u_0\Vert_{L^\infty})>0,  
   \llabel{time0}
\end{align}
and 
   \begin{align}
     M_0 = M_0(T_0,\Vert u_0\Vert_{4},\Vert u_0\Vert_{2,\infty},\Vert \nabla u_0\Vert_{L^\infty})>0,  
     \llabel{bound}
   \end{align}
  with the following property:
If $u \in C([0,T_0];H^5(\Omega))$
is a solution of the Euler equations \eqref{NSE0}--\eqref{hnavierbdry}
on $[0,T_0]$, then
   \begin{align}
    \sup_{[0,T_0]}(\Vert u(t)\Vert_{4}
    +\Vert u(t)\Vert_{2,\infty}
    +\Vert \nabla u(t)\Vert_{L^\infty})   
    \le M_0.
    \label{ap}
   \end{align}
\end{Proposition}
\colb

Based on the three estimates given by Propositions~\ref{P.Con}, \ref{P.Pre}, and \ref{P.Inf},
we prove Proposition~\ref{P.Ap} in the following section.
First, in Proposition~\ref{P.Con}, we estimate $u$ in~$H^4_\cco(\Omega)$.
Next, in Proposition~\ref{P.Pre}, we present estimates for $\nabla p$ and~$D^2 p$.
Finally, in Proposition~\ref{P.Inf}, we establish the $L^\infty$ bounds.
Then, the proof of Proposition~\ref{P.Ap} follows by a Gr\"onwall argument.
Having the a~priori estimates, we begin the construction in Section~\ref{sec.main}
by regularizing the initial data. Then, we obtain approximate solutions by utilizing the classical 
well-posedness theory for Euler equations.
The maximal time of existence of these solutions depends on the approximation level
and may shrink to $0$ as we progress further in the approximation.
However, since the Lipschitz norm of $u$ on $[0,T_0]$ remains bounded, we prove that this does not happen. 
Finally, we pass to the limit by proving that the sequence of approximate
solutions is Cauchy in~$L^\infty_tL^2_x$.
 
 \startnewsection{A~priori estimates}{sec.apri}
In this section, we prove Proposition~\ref{P.Ap} by establishing conormal, pressure, and $L^\infty$
estimates. Manipulations leading to the a~priori estimates are justified
under the assumption $u \in C([0,T_0];H^5(\Omega))$. We note that $H^5(\Omega)$
is not the weakest space to prove \eqref{ap}, and the regularity level needed for
the a~priori estimates does not affect our main result.
  
\subsection{Conormal derivative estimates}
First, we estimate $\Vert u\Vert_{4}$ in
terms of $\Vert u\Vert_{W^{1,\infty}}$, $\Vert u\Vert_{2,\infty}$,
and $\Vert \nabla p\Vert_{3}$.
 
\cole
\begin{Proposition}
   \label{P.Con}
   Let $u \in C([0,T];H^5(\Omega))$
   be a solution of the Euler equations \eqref{NSE0}--\eqref{hnavierbdry}
   for some $T>0$.
Then we have the inequality
   \begin{align}
     \Vert u(t)\Vert_{4}^2
     \lec
     \Vert u_0\Vert_{4}^2
     +\int_0^t
       \Bigl(
         \Vert u(s)\Vert_{4}^2 
         (\Vert u(s)\Vert_{W^{1,\infty}}+
         \Vert u(s)\Vert_{2,\infty}+1
         )+\Vert u(s)\Vert_{4}\Vert \nabla p(s)\Vert_{3}
       \Bigr)\,ds
     ,
     \label{EQ.Con}
   \end{align}
for $t \in [0,T]$. 
\end{Proposition}
\colb
 
\begin{proof}[Proof of Proposition~\ref{P.Con}] 
We prove \eqref{EQ.Con} by showing
   \begin{align}
     \Vert u(t)\Vert_{m}^2
     \lec
     \Vert u_0\Vert_{m}^2
     +\int_0^t
       \Bigl(
         \Vert u(s)\Vert_{m}^2 
         (\Vert u(s)\Vert_{W^{1,\infty}}+
         \Vert u(s)\Vert_{2,\infty}+1
         )
	 +1_{m\geq1}\Vert u(s)\Vert_{m}\Vert \nabla p(s)\Vert_{m-1}
       \Bigr)\,ds
   \label{EQ03}
   \end{align}
for $m=0,1,2,3,4$ by induction on~$m$.
The standard $L^2$ estimate reads
$
      \frac{1}{2}\frac{d}{dt}\Vert u\Vert_{L^2}^2=0
$,
which implies \eqref{EQ03} for~$m=0$.
We now focus on the final step of the induction.
Therefore, our goal is to establish \eqref{EQ.Con}
under the assumption that
\eqref{EQ03} holds for $m=3$.
We begin by examining 
$Z^\alpha = Z^{\tilde{\alpha}}_\hh$ for $|\alpha|=4$;
this corresponds to the scenario where all the conormal derivatives are horizontal.
We apply $Z^\alpha = Z^{\tilde{\alpha}}_\hh$ to \eqref{NSE0}$_1$, obtaining
   \begin{align}
     \partial_t Z^{\tilde{\alpha}}_\hh u
     + u \cdot \nabla Z^{\tilde{\alpha}}_\hh u
     + \nabla Z^{\tilde{\alpha}}_\hh p
     =
     u \cdot \nabla Z^{\tilde{\alpha}}_\hh u
     -
     Z^{\tilde{\alpha}}_\hh(u \cdot \nabla u)
     ,
     \label{EQ01}
   \end{align}
where we used $Z_h \nabla = \nabla Z_h$.
We multiply \eqref{EQ01} with $Z^{\tilde{\alpha}}_\hh u$ and
integrate, which gives
\begin{align}
     \frac{1}{2}\frac{d}{dt}\Vert Z^{\tilde{\alpha}}_\hh u\Vert_{L^2}^2
     =
     -
     ( u \cdot \nabla Z^{\tilde{\alpha}}_\hh u
     -Z^{\tilde{\alpha}}_\hh(u \cdot \nabla u)
     ,Z^{\tilde{\alpha}}_\hh u)
     ,
     \label{EQ54}
   \end{align} 
since the pressure term  vanishes due to the divergence-free assumption.
We rewrite the commutator term in \eqref{EQ54} as
   \begin{align}
     u \cdot \nabla Z^{\tilde{\alpha}}_\hh u
     -
     Z^{\tilde{\alpha}}_\hh(u \cdot \nabla u)
     =
     -\sum_{1\le |\tilde{\beta}| \le 4} 
     {\tilde{\alpha} \choose \tilde{\beta}}
     (Z^{\tilde{\beta}}_\hh u_\hh 
     \cdot 
     \nabla_\hh Z^{\tilde{\alpha}-\tilde{\beta}}_\hh u
     + 
     Z^{\tilde{\beta}}_\hh u_3 
     \partial_z Z^{\tilde{\alpha}-\tilde{\beta}}_\hh u)
     .
     \label{EQ45}                              
   \end{align}
The first term inside the parentheses on the right-hand side of \eqref{EQ45}
is estimated
by taking the $L^\infty$ norm of the factor with fewer derivatives.
Thus, we have
   \begin{align}
     \Vert 
     Z^{\tilde{\beta}}_\hh u_\hh 
     \cdot 
     \nabla_\hh Z^{\tilde{\alpha}-\tilde{\beta}}_\hh u
     \Vert_{L^2}
     \lec
     \Vert u\Vert_{2,\infty}\Vert u\Vert_{4}
     ,
     \label{EQ30}
   \end{align} 
for all $1\leq |\tilde \beta|\leq 4$.
To estimate the remaining terms in \eqref{EQ45}
when $|\tilde{\beta}|\neq 4$, we conormalize $\partial_z$ by dividing and multiplying with~$\varphi$.
   Then, we employ \eqref{EQ.int}
   by choosing $(f,g)=(Z_\hh \frac{u_3}{\varphi}, Z_\hh Z_3 u)$
   for $|\tilde{\beta}|=2$ and~\eqref{EQ.u3} to estimate $Z_\hh \frac{u_3}{\varphi}$.
Summarizing, we arrive at
   \begin{equation}
     \label{EQ143}
     \left\Vert Z_\hh^{\tilde{\beta}} \frac{u_3}{\varphi} Z_3 Z_\hh^{\tilde{\alpha} - {\tilde{\beta}}} u\right\Vert_{L^2}  
     \lec 
     \begin{cases}
       \bigl\Vert Z \frac{u_3}{\varphi}\bigr\Vert_{L^\infty}\Vert u\Vert_{4}
       \lec \Vert u\Vert_{2,\infty}\Vert u\Vert_{4}, & |{\tilde{\beta}}|=1 \\
       \bigl\Vert Z_\hh \frac{u_3}{\varphi}\bigr\Vert_{L^\infty}\Vert Z_\hh Z_3 u\Vert_{2}
       +\left\Vert Z_\hh \frac{u_3}{\varphi}\right\Vert_{2}\Vert Z_\hh Z_3 u\Vert_{L^\infty}
       \lec \Vert u\Vert_{4}\Vert u\Vert_{2,\infty}, & |{\tilde{\beta}}|=2 \\
       \left\Vert Z_\hh \frac{u_3}{\varphi}\right\Vert_{2}\Vert Z_\hh Z_3 u\Vert_{L^\infty}
       \lec \Vert u\Vert_{4}\Vert u\Vert_{2,\infty}, &   |{\tilde{\beta}}|=3,
     \end{cases}
   \end{equation}
   Lastly, for the case $|\tilde{\beta}| = 4$, taking the uniform norm of $\partial_z u$ gives 
   \begin{align}
     \Vert Z_\hh^{\tilde{\beta}} u_3 \partial_z Z_\hh^{\tilde{\alpha} - {\tilde{\beta}}} u\Vert_{L^2}  
     \lec 
     \Vert u\Vert_{4}\Vert \partial_z u\Vert_{L^\infty}
     \comma   |{\tilde{\beta}}|=4
     .\label{EQ33} 
   \end{align}
Collecting \eqref{EQ54}--\eqref{EQ33} yields
   \begin{align}
       \frac{d}{dt}\Vert Z^{\tilde{\alpha}}_\hh u\Vert_{L^2}^2 
       \lec
       \Vert u\Vert_{4}
       (\Vert u\Vert_{4}+ 
       \Vert 
       u \cdot \nabla Z^{\tilde{\alpha}}_\hh u
       -
       Z^{\tilde{\alpha}}_\hh(u \cdot \nabla u)
       \Vert_{L^2}
       )
       \lec
       \Vert u\Vert_{4}^2
       (\Vert u\Vert_{2,\infty}+\Vert u\Vert_{W^{1,\infty}}+1)
       ,  
       \llabel{EQ38}   
   \end{align}
   from where, integrating in time, 
   \begin{align}
     \Vert Z^{\tilde{\alpha}}_\hh u(t)\Vert_{L^2}^2 
     \lec
     \Vert Z^{\tilde{\alpha}}_\hh u_0\Vert_{L^2}^2
     +\int_0^t \Vert u(s)\Vert_{4}^2
     (\Vert u(s)\Vert_{2,\infty}+\Vert u(s)\Vert_{W^{1,\infty}}+1)\,ds
     ,\llabel{EQ04}
   \end{align}
   for~$t \in [0,T]$.
   
We now focus on the case where $Z^\alpha = Z^{\tilde{\alpha}}_\hh Z^k_3$ 
for $1\le k \le 4$. We apply $Z^\alpha$ to \eqref{NSE0}, obtaining
   \begin{align}
     Z^\alpha u_t
     + u\cdot \nabla Z^\alpha u
     = 
     u\cdot \nabla Z^\alpha u - Z^\alpha(u\cdot \nabla u) 
     - Z^\alpha \nabla p
     \label{EQ46}
   \end{align}
and then test \eqref{EQ46} with $Z^\alpha u$, which leads to
   \begin{align}
     \frac{1}{2}\frac{d}{dt}\Vert Z^\alpha u\Vert_{L^2}^2
   =   
   (u\cdot \nabla Z^\alpha u - Z^\alpha(u\cdot \nabla u), Z^\alpha u )
   - (Z^\alpha \nabla p, Z^\alpha u)
   .
     \label{EQ55}
   \end{align}
The treatment of the convective term is almost identical to \eqref{EQ45}--\eqref{EQ33}, the only difference being
is the commutator of $\partial_z$ and $Z_3$, which is of lower order.
To make this precise, we rewrite the convection term as
   \begin{align}
     \begin{split}
       &
       u \cdot \nabla Z^\alpha u - Z^\alpha(u \cdot \nabla u)
       =
       u \cdot \nabla Z^\alpha u - u \cdot Z^\alpha \nabla u
       -
       \sum_{1 \le |\beta|\le |\alpha|}
       {\alpha \choose \beta}
       Z^\beta u \cdot Z^{\alpha-\beta} \nabla u
       \\&\indeq
       =
       u_3 \partial_z Z^\alpha u - u_3 Z^\alpha \partial_z u    
       -
       \sum_{1\le |\beta|\le |\alpha|}
       {\alpha \choose \beta}
       Z^\beta u \cdot Z^{\alpha-\beta} \nabla u
       \\&\indeq
       =
       -
       \sum_{j=0}^{k-1}
       \tilde{c}^k_{j,\varphi} u_3 \partial_z Z^{\tilde{\alpha}}_\hh Z^j_3 u  
       -
       \sum_{1\le |\beta|\le |\alpha|}
       {\alpha \choose \beta}
       Z^\beta u \cdot Z^{\alpha-\beta} \nabla u
       \\&\indeq
       = I_1 + I_2
       .
       \llabel{EQ13}        
     \end{split}
   \end{align}
   We treat $I_1$ by conormalizing $\partial_z$, i.e., we write
   \begin{align}
     I_1\lec
     \sum_{j=0}^{k-1}
     \left\Vert \tilde{c}^k_{j,\varphi} \frac{u_3}{\varphi} Z^{\tilde{\alpha}}_\hh Z^{j+1}_3 u\right\Vert_{L^2}  
     \lec
     \left\Vert \frac{u_3}{\varphi}\right\Vert_{L^\infty}\Vert u\Vert_{4}
     \lec
     \Vert u\Vert_{1,\infty}\Vert u\Vert_{4}
     ,
     \llabel{EQ56}
   \end{align}
utilizing~\eqref{EQ.u3}.
Now, we expand $I_2$ as  
   \begin{align}
     I_2 = -
     \sum_{1\le |\beta| \le |\alpha|}
     {\alpha \choose \beta} 
     (Z^{\beta} u_\hh 
     \cdot 
     \nabla_\hh Z^{\alpha-\beta} u
     + 
     Z^{\beta} u_3 
     Z^{\alpha-\beta} \partial_z u)
     = I_{21} + I_{22}
     .
     \llabel{EQ57}
   \end{align}
To control the sum $I_{21}$, we employ the bounds in~\eqref{EQ30}.
For $I_{22}$, we repeat \eqref{EQ143}--\eqref{EQ33} and obtain lower-order terms for the case~$|\beta|\neq 4$.
 Namely, due to the identities
   \begin{align}
     \frac{1}{\varphi} {Z_3 u_3} = \nabla_\hh \cdot u_\hh
    \aand
     \frac{1}{\varphi} Z_\hh = Z_\hh \frac{1}{\varphi} 
      ,\llabel{EQ32}
   \end{align}
   it suffices to commute $\partial_z$ and $Z_3$,
   and this results in lower order terms given by
   \eqref{EQL02}(i).
   In summary, we arrive at
   \begin{align}
     (u \cdot \nabla Z^\alpha u - Z^\alpha(u \cdot \nabla u),
     Z^\alpha u)
     \lec 
     \Vert u\Vert_{4}
     (\Vert u\Vert_{2,\infty}+\Vert u\Vert_{W^{1,\infty}}).
     \label{EQ58}
   \end{align}
The negative of the pressure term in~\eqref{EQ55} is expanded as
   \begin{align}
     (Z^\alpha \nabla p, Z^\alpha u)
     =
     (Z^\alpha \nabla_\hh p, Z^\alpha u_\hh)
     +
     (Z^\alpha \partial_z p, Z^\alpha u_3) 
     .
     \label{EQ47}
   \end{align}
   Although $Z_3$ and $\partial_z$ do not commute, 
   it is possible to reduce the number of conormal derivatives on the pressure by one. 
   Employing Lemma~\ref{L01}(i) for the last term in \eqref{EQ47}, we have
   \begin{align}
     (Z^k_3 \partial_z Z^{\tilde{\alpha}}_\hh p, Z^\alpha u_3)
     =
     (\partial_z Z^k_3  Z^{\tilde{\alpha}}_\hh p, Z^\alpha u_3)
     +
     \sum_{j=0}^{k-1}
     (c^k_{j,\varphi} \partial_z Z^j_3  Z^{\tilde{\alpha}}_\hh p, Z^\alpha u_3) 
     .
     \label{EQ48}
   \end{align}
   The sum above consists of lower-order terms that are controlled by
   $\Vert \nabla p\Vert_{3}\Vert u\Vert_{4}$.
For the first term on the right, integrating by parts gives
   \begin{align}
     (\partial_z Z^k_3 Z^{\tilde{\alpha}}_\hh p, Z^\alpha u_3) 
     =
     (Z^\alpha p, \partial_z Z^\alpha u_3)
     =
     (Z^\alpha p, Z^\alpha \partial_z  u_3)
     -
     \sum_{j=0}^{k-1}
     (c^k_{j,\varphi} Z^\alpha p, Z^j_3 Z^{\tilde{\alpha}}_\hh \partial_z u_3) 
     ,
     \label{EQ49}
   \end{align}
    since the boundary term vanishes due to $Z_3 = 0$
   on~$\partial \Omega$. 
Once again, the sum above consists of lower order terms, which are controlled by $\Vert \nabla p\Vert_{3}\Vert u\Vert_{4}$.
Now, we collect \eqref{EQ47}--\eqref{EQ49} and use $Z^\alpha \nabla \cdot u = 0$ to obtain
   \begin{align}
     -
     (Z^\alpha \nabla p, Z^\alpha u)
     \lec 
     \Vert \nabla p\Vert_{3}\Vert u\Vert_{4}
     .
     \label{EQ50}
   \end{align}
Combining \eqref{EQ55}, \eqref{EQ58}, \eqref{EQ50}, and integrating in
time leads to
   \begin{align}
       \Vert Z^\alpha u(t)\Vert_{L^2}^2
       \lec
       \Vert Z^\alpha u_0\Vert_{L^2}^2
       +\int_0^t
       \Bigl(\Vert u(s)\Vert_{4}^2 
       (\Vert u(s)\Vert_{W^{1,\infty}}+
       \Vert u(s)\Vert_{2,\infty}+1
       ) 
       + \Vert u(s)\Vert_{4}\Vert \nabla p(s)\Vert_{3}
       \Bigr) \,ds
       ,
       \llabel{EQ05}  
   \end{align} 
completing the induction step.
 \end{proof}
 
\subsection{Pressure Estimates}
 
Here, we present the $H^3_\cco(\Omega)$ bounds on $\nabla p$ and~$D^2 p$.
The latter estimate is needed when we bound~$\Vert u\Vert_{2,\infty}$;
see Section~\ref{sec33} below.
 
\cole
\begin{Proposition}
\label{P.Pre}
Let $u \in C([0,T];H^5(\Omega))$ be a solution of \eqref{NSE0}--\eqref{hnavierbdry}
for some $T>0$. 
Then we have the inequality
  \begin{align}
   \Vert D^2 p(t)\Vert_{3}+\Vert \nabla p(t)\Vert_{3}
   \lec 
   \Vert u(t)\Vert_{4}(\Vert u(t)\Vert_{2,\infty}+\Vert \nabla u(t)\Vert_{L^\infty}+1)
   ,
   \label{EQ.Pre}
  \end{align}
for $t\in [0,T]$. 
 \end{Proposition}
 \colb
 
To prove Proposition~\ref{P.Pre}, we first consider $Z^\alpha= Z^{\tilde{\alpha}}_\hh $
and then separately consider the  general case
$Z^\alpha = Z^{\tilde{\alpha}}_\hh Z^k_3$
using the induction on~$k$.

\begin{proof}[Proof of Proposition~\ref{P.Pre}]

For $u$ and $T>0$ as in the statement, we claim
that
 \begin{align}
  \Vert D^2 Z^{\tilde{\alpha}}_\hh p(t)\Vert_{L^2}+
  \Vert \nabla Z^{\tilde{\alpha}}_\hh p(t)\Vert_{L^2}
  \lec 
  \Vert u(t)\Vert_{4}(\Vert u(t)\Vert_{2,\infty}+\Vert \nabla u(t)\Vert_{L^\infty}+1)
  ,
  \label{EQ.hPre}
 \end{align}
 for all $t \in [0,T]$ and $0 \le |\tilde{\alpha}|\le 3$.
We solve for $-\Delta p$ using incompressibility in \eqref{NSE0}, obtaining
  \begin{align}
  -\Delta p = \partial_i u_j \partial_j u_i
  ,\label{pre} 
  \end{align}
while on $\partial \Omega$, employing \eqref{hnavierbdry}, we get
  \begin{align}
  \nabla p \cdot n = -\partial_z p
  =0
  \label{pre.bdry}
  \end{align}
We only consider the case $|\tilde{\alpha}|=3$ as other cases follow the analogous proof.
Applying $Z^{\tilde{\alpha}}_\hh $ to \eqref{pre}--\eqref{pre.bdry}, we obtain
 \begin{align}
  -\Delta Z^{\tilde{\alpha}}_\hh p = Z^{\tilde{\alpha}}_\hh (\partial_i u_j \partial_j u_i)
  ,
  \llabel{hpre}
 \end{align}
 coupled with the Neumann boundary condition
 \begin{align}
  \nabla Z^{\tilde{\alpha}}_\hh p \cdot n
  = 0.
  \llabel{hpre.bdry}
 \end{align}
 Employing the standard elliptic theory (see, for example,~\cite{Gr}), we conclude that
 \begin{align}
  \Vert D^2 Z^{\tilde{\alpha}}_\hh p\Vert_{L^2}
  +
  \Vert \nabla Z^{\tilde{\alpha}}_\hh p\Vert_{L^2}
  \lec
  \Vert Z^{\tilde{\alpha}}_\hh (\partial_i u_j \partial_j u_i)\Vert_{L^2}
  .\label{EQ87}
 \end{align}
To establish \eqref{EQ.hPre}, we only estimate
$Z^{\tilde{\alpha}}_\hh (Z_\hh u_\hh Z_\hh u_\hh )$ 
 and 
 $Z^{\tilde{\alpha}}_\hh (Z_\hh u_3 \partial_z u_\hh)$.
Indeed,
 due to the incompressibility, we note that
 $\partial_z u_3 \partial_z u_3$ consists only of horizontal derivatives.
 When there are no normal derivatives, employing 
 \eqref{EQ.int} yields
 \begin{align}
  \Vert Z^{\tilde{\alpha}}_\hh (Z_\hh u_\hh Z_\hh u_\hh )\Vert_{L^2}
  \lec
  \Vert u\Vert_{4}\Vert u\Vert_{1,\infty}.
  \label{EQ89}
 \end{align}
 On the other hand, rewriting the term involving $\partial_z$ gives
 \begin{align}
  Z^{\tilde{\alpha}}_\hh (Z_\hh u_3 \partial_z u_\hh)=
  \sum_{0\le |\tilde{\beta}|\le |\tilde{\alpha}|}
  {\alpha \choose \beta}
  Z^{\tilde{\beta}}_\hh Z_\hh u_3
  Z^{\tilde{\alpha}-\tilde{\beta}}_\hh \partial_z u_\hh
  .\label{EQ90}
 \end{align}
 Assuming $|\tilde{\beta}|=3$, we have
 \begin{align}
  \Vert Z_\hh^{\tilde{\beta}} Z_\hh u_3 \partial_z Z_\hh^{\tilde{\alpha} - {\tilde{\beta}}} u_\hh\Vert_{L^2}\lec
  \Vert u\Vert_{4}\Vert \nabla u\Vert_{L^\infty}
  .\label{EQ91a}
 \end{align} 
 When $|\tilde{\beta}|\neq 3$,
 we conormalize $\partial_z$ upon dividing and multiplying by $\varphi$ to obtain
 \begin{equation}
  \label{EQ91b}
  \left\Vert Z_\hh^{\tilde{\beta}} Z_\hh \frac{u_3}{\varphi} 
  Z_3 Z_\hh^{\tilde{\alpha} - {\tilde{\beta}}} u_\hh\right\Vert_{L^2}
  \lec
  \begin{cases}
   \bigl\Vert Z_\hh \frac{u_3}{\varphi}\bigr\Vert_{L^\infty}\Vert Z^{\tilde{\alpha}-\tilde{\beta}}_\hh Z_3 u\Vert_{L^2}
   \lec \Vert u\Vert_{4}\Vert u\Vert_{2,\infty}, & |{\tilde{\beta}}|=0
   \\
   \left(\bigl\Vert Z_\hh \frac{u_3}{\varphi}\bigr\Vert_{L^\infty}\Vert Z_\hh Z_3 u\Vert_{2}
   +\left\Vert Z_\hh \frac{u_3}{\varphi}\right\Vert_{2} \Vert Z_\hh Z_3 u\Vert_{L^\infty}\right)
   \lec\Vert u\Vert_{4}\Vert u\Vert_{2,\infty}, & |{\tilde{\beta}}|=1
   \\
   \Vert u\Vert_{4}\Vert u\Vert_{2,\infty}, & |{\tilde{\beta}}|=2,
  \end{cases}
 \end{equation}
 where, when $|\tilde{\beta}|= 1$, we have employed \eqref{EQ.int}
 by choosing $f$ as $Z_\hh \frac{u_3}{\varphi}$ and $g$ as $Z_\hh Z_3 u$. 
 
Collecting \eqref{EQ87}--\eqref{EQ91b}
gives \eqref{EQ.hPre} for~$|\tilde{\alpha}|=3$.
Similar estimates can be used to bound lower order conormal derivatives
of $\nabla p$ and $D^2 p$, i.e., 
we may repeat \eqref{EQ87}--\eqref{EQ91b} for $|\tilde{\alpha}|\le 2$ and 
conclude that~\eqref{EQ.hPre} holds.

Next, we proceed to the case $Z^\alpha = Z^{\tilde{\alpha}}_\hh Z^k_3$ and $k \ne 0$. We only focus on conormal derivatives of order three, i.e., $|\alpha|=3$. 
To estimate $\Vert D^2 Z^\alpha p\Vert_{L^2}$ and
 $\Vert \nabla Z^\alpha p\Vert_{L^2}$, we employ
 induction on $0\le k \le 3$ such that 
 $k + |\tilde{\alpha}| = 3$. 
 Since $|\alpha|=3$ is fixed, the total number of $Z_3$ in $Z^\alpha$ increases on the inductive step,
 whereas the total number of horizontal derivatives decreases. 
 When $k=0$, \eqref{EQ.hPre} yields the base step of the induction.
 Now, the inductive hypothesis is 
 \begin{align}
   \Vert D^2 Z^{\tilde{\beta}}_\hh Z^{k}_3 p\Vert_{L^2}+
   \Vert \nabla Z^{\tilde{\beta}}_\hh Z^{k}_3 p\Vert_{L^2}
   \lec
   \Vert u\Vert_{4}(\Vert u\Vert_{2,\infty}+\Vert \nabla u\Vert_{L^\infty}+1) 
   \comma |\tilde{\beta}| + k = 3.
   \label{EQ.pre.k-1} 
 \end{align}
  Therefore, our goal is to establish that
  \begin{align}
    \Vert D^2 Z^{\tilde{\alpha}}_\hh Z^{k+1}_3 p\Vert_{L^2}+
    \Vert \nabla Z^{\tilde{\alpha}}_\hh Z^{k+1}_3 p\Vert_{L^2}
    \lec
    \Vert u\Vert_{4}(\Vert u\Vert_{2,\infty}+\Vert \nabla u\Vert_{L^\infty}+1) 
    \comma |\tilde{\alpha}| + k + 1 = 3
    .
    \label{EQ.pre.k}
  \end{align}
We apply $Z^\alpha = Z^{\tilde{\alpha}}_\hh Z^{k+1}_3$ to \eqref{pre}, obtaining
  \begin{align}
   -\Delta Z^\alpha p  
   = Z^\alpha(\partial_i u_j \partial_j u_i)
   + Z^\alpha \Delta p -\Delta Z^\alpha p
   ,
   \llabel{kpre}  
  \end{align}
with the boundary condition
 \begin{align}
  \nabla Z^\alpha p \cdot n
  = -\partial_z Z^{\tilde{\alpha}}_\hh Z^{k+1}_3 p
  = -Z^{\tilde{\alpha}}_\hh \sum_{j=0}^{k}
  \tilde{c}_{j,\varphi}^{k+1} Z_3^j \partial_z p 
  = -\tilde{c}_{0,\varphi}^{k+1} Z^{\tilde{\alpha}}_\hh \partial_z p 
  = 0
  ,\llabel{kpre.bdry}
 \end{align}
 where we have employed \eqref{pre.bdry} and $Z_3=0$ on~$\partial \Omega$.
As in \eqref{EQ87}, we use the elliptic estimates to get
 \begin{align}
  \Vert D^2 Z^\alpha p\Vert_{L^2}+
  \Vert \nabla Z^\alpha p\Vert_{L^2}
  \lec
  \Vert Z^\alpha(\partial_i u_j \partial_j u_i)\Vert_{L^2}
  +
  \Vert Z^\alpha \Delta p -\Delta Z^\alpha p\Vert_{L^2}
  .\label{EQ93}
 \end{align}
 For the quadratic term in \eqref{EQ93},
 we proceed as in \eqref{EQ89}--\eqref{EQ91b} and
employ Lemma~\ref{L01} to commute $\partial_z$ and $Z_3$, with the only difference 
being the lower order terms. Thus, we conclude
 \begin{align}
  \Vert Z^\alpha(\partial_i u_j \partial_j u_i)\Vert_{L^2}
  \lec 
  \Vert u\Vert_{4}(\Vert u\Vert_{2,\infty}+\Vert \eta\Vert_{L^\infty}+1)
  .
  \label{EQ94}
 \end{align}
 Now, we consider the commutator term for the pressure.
 Employing Lemma~\ref{L01}
 and using 
 $(\tilde{c}_{k+1,\varphi}^{k+1})' = (1)' = 0$,
 we rewrite this term as
 \begin{align}
  \begin{split}
   Z^\alpha \Delta p -\Delta Z^\alpha p
   &=Z^\alpha \partial_{zz} p - \partial_{zz}Z^\alpha p
   = Z^\alpha \partial_{zz} p -
   \sum_{j=0}^{k+1} \sum_{l=0}^{j}
   \tilde{c}^j_{l,\varphi} \tilde{c}^{k+1}_{j,\varphi} Z^l_3 Z^{\tilde{\alpha}}_\hh \partial_{zz} p 
   +
   \sum_{j=0}^{k+1}
   (\tilde{c}^{k+1}_{j,\varphi})' Z^j_3 Z^{\tilde{\alpha}}_\hh \partial_z p 
   \\&=
   -\sum_{j=0}^{k} \sum_{l=0}^{j}
   \tilde{c}^j_{l,\varphi} \tilde{c}^{k+1}_{j,\varphi} Z^l_3 Z^{\tilde{\alpha}}_\hh \partial_{zz} p 
   -
   \sum_{l=0}^{k}
   \tilde{c}^{k+1}_{l,\varphi} Z^l_3 Z^{\tilde{\alpha}}_\hh \partial_{zz} p 
   +
   \sum_{j=0}^{k}
   (\tilde{c}^{k+1}_{j,\varphi})' Z^j_3 Z^{\tilde{\alpha}}_\hh \partial_z p
   \\&= I_1 + I_2 + I_3     
   .\label{EQ95}
  \end{split}
 \end{align}
 We only focus on the highest-order terms in \eqref{EQ95}. 
 Starting with $I_3$, 
 we consider $\partial_z Z^{\tilde{\alpha}}_\hh Z^{k}_3 p$ which we estimate
 using~\eqref{EQ.pre.k-1}.
 Next, for both $I_1$ and $I_2$,
 $\tilde{c}_{k,\varphi}^{k+1} Z^{\tilde{\alpha}}_\hh Z^{k}_3 \partial_{zz} p$
 is the highest-order term. 
 We employ \eqref{pre} to expand this term as
 \begin{align}
  \tilde{c}_{k,\varphi}^{k+1} Z^{\tilde{\alpha}}_\hh Z^{k}_3 \partial_{zz} p
  =
  -\tilde{c}_{k,\varphi}^{k+1} Z^{\tilde{\alpha}}_\hh Z^{k}_3 \Delta_\hh p
  -\tilde{c}_{k,\varphi}^{k+1} Z^{\tilde{\alpha}}_\hh Z^{k}_3 (\partial_i u_j \partial_j u_i)
  .\llabel{EQ97}
 \end{align}
 Recalling that $k+|\alpha| =2$, 
 the term $Z^{\tilde{\alpha}}_\hh Z^{k}_3 (\partial_i u_j \partial_j u_i)$
 is estimated as in \eqref{EQ89}--\eqref{EQ91b}.
 In addition, 
 we use \eqref{EQ.pre.k-1} and estimate $Z^{\tilde{\alpha}}_\hh Z^{k}_3 \Delta_\hh p$ as
 \begin{align}
  \Vert Z^{\tilde{\alpha}}_\hh Z^{k}_3 \Delta_\hh p\Vert_{L^2}
  \lec
  \Vert D^2 Z^{\tilde{\alpha}}_\hh Z^{k}_3 p\Vert_{L^2}
  \lec 
  \Vert u\Vert_{4}(\Vert u\Vert_{2,\infty}+\Vert \nabla u\Vert_{L^\infty}+1)
  .\llabel{EQ98}
 \end{align}
 Utilizing these estimates for the other terms in $I_1, I_2$, and $I_3$,
 we arrive at
 \begin{align}
  \Vert Z^\alpha \Delta p -\Delta Z^\alpha p\Vert_{L^2}
  \lec \Vert u\Vert_{4}(\Vert u\Vert_{2,\infty}+\Vert \nabla u\Vert_{L^\infty}+1)
  .\label{EQ99}
 \end{align}
 Collecting \eqref{EQ93}--\eqref{EQ94}, and \eqref{EQ99},
 we conclude \eqref{EQ.pre.k}
 and the proof of Proposition~\ref{P.Pre}.
\end{proof}

\subsection{Uniform bounds}
\label{sec33}
Now, we establish bounds for $\Vert \nabla u\Vert_{L^\infty}$ and~$\Vert u\Vert_{2,\infty}$.

\cole
\begin{Proposition}
 \label{P.Inf}
 Let $u \in C([0,T];H^5(\Omega))$ be a solution of \eqref{NSE0}--\eqref{hnavierbdry}
 for some $T>0$.
 Then we have the inequality
  \begin{align}
    \Vert u(t)\Vert_{2,\infty}^2+\Vert \nabla u(t)\Vert_{L^\infty}^2
    \lec 
    \Vert u_0\Vert_{2,\infty}^2+\Vert \nabla u_0\Vert_{L^\infty}^2
    +
    \int_0^t
    \Bigl(
     (\Vert u\Vert_{2,\infty}+\Vert \nabla u\Vert_{L^\infty}+1)^3
      +\Vert u\Vert_{2,\infty}\Vert D^2 p\Vert_{3}
    \Bigr)
    \,ds
    ,
    \label{EQ.Inf1}
  \end{align}
  for $t \in [0,T]$.
  \end{Proposition}
\colb

\begin{proof}[Proof of Proposition~\ref{P.Inf}]
 
 To estimate the Lipschitz norm of $u$, we use the vorticity formulation 
 \begin{align}
   \omega_t + u\cdot \nabla \omega = \omega \cdot \nabla u
   ,
   \label{vorticity}
 \end{align}
as
 \begin{align}
   \Vert \nabla u\Vert_{L^\infty}
    \lec
     \Vert \omega\Vert_{L^\infty}+\Vert u\Vert_{1,\infty}
     ,\label{EQ300}
 \end{align}
 showing that the normal derivative of $u$ is
 controlled by the vorticity and the conormal derivatives of~$u$.
 Now, using the maximum principle for \eqref{vorticity}, it follows that
 \begin{align}
  \frac{d}{dt}\Vert \omega\Vert_{L^\infty}
  \lec
  \Vert \omega\Vert_{L^\infty}\Vert \nabla u\Vert_{L^\infty}
  \lec
  \Vert \omega\Vert_{L^\infty}
   (\Vert \omega\Vert_{L^\infty}+\Vert u\Vert_{1,\infty})
  ,
  \label{EQ101}
 \end{align}
 where we have employed \eqref{EQ300}.
 To estimate $\Vert u\Vert_{2,\infty}$,
 we solely focus on $Z^\alpha u$ for $|\alpha|=2$ and note in passing that
 the case $|\alpha|=1$ can be treated in a similar fashion.
 Recalling that 
 $Z^\alpha = Z^{\tilde{\alpha}}_\hh Z^k_3$ and $|\tilde{\alpha}| + k =2$,
 it follows that $Z^\alpha u$ is a solution of
 \begin{align}
  Z^\alpha u_t
  + u\cdot \nabla Z^\alpha u
  = 
  u\cdot \nabla Z^\alpha u - Z^\alpha(u\cdot \nabla u) 
  - Z^\alpha \nabla p
  .
  \label{EQ102}
 \end{align}
Next, with $p>4$, 
test \eqref{EQ102} with 
$Z^\alpha u |Z^\alpha u|^{p-2}$ and note that
the left-hand side of \eqref{EQ102} leads to the term
$   \frac{1}{p}\frac{d}{dt}\Vert Z^\alpha u\Vert_{L^p}^p$.
Proceeding with the right-hand side,
it follows that
 \begin{align}
  (u\cdot \nabla Z^\alpha u - Z^\alpha(u\cdot \nabla u) 
  - Z^\alpha \nabla p,
  Z^\alpha u |Z^\alpha u|^{p-2})
  \lec 
  \Vert u\cdot \nabla Z^\alpha u - Z^\alpha(u\cdot \nabla u) 
  - Z^\alpha \nabla p\Vert_{L^p}
  \Vert Z^\alpha u\Vert_{L^p}^{p-1}
  ,
  \llabel{EQ104}
 \end{align}
and thus we conclude that
 \begin{align}
   \frac{1}{p}\frac{d}{dt}\Vert Z^\alpha u\Vert_{L^p}^p
   \lec
   \Vert Z^\alpha u\Vert_{L^p}^{p-1}
   (\Vert u\cdot \nabla Z^\alpha u - Z^\alpha(u\cdot \nabla u) 
   - Z^\alpha \nabla p\Vert_{L^p})
   .\llabel{EQ113}
  \end{align}
 Now, we divide both sides by $\Vert Z^\alpha u\Vert_{L^p}^{p-2}$
 and pass to the limit as~$p\to \infty$. Finally, integrating in time yields
 \begin{align}
  \begin{split}
   \Vert Z^\alpha u(t)\Vert_{L^\infty}^2
   \lec&
   \Vert u_0\Vert_{2,\infty}+\int_0^t 
   \Vert u\Vert_{2,\infty}
   (\Vert u\cdot \nabla Z^\alpha u - Z^\alpha(u\cdot \nabla u)\Vert_{L^\infty}
   +\Vert Z^\alpha \nabla p\Vert_{L^\infty})
   \,ds         
   .\label{EQ114}
  \end{split}
 \end{align}
It remains to estimate the right-hand side of \eqref{EQ114}.
We begin with the commutator term and write
 \begin{align}
   \begin{split}
  u\cdot \nabla Z^\alpha u - Z^\alpha(u\cdot \nabla u)
  &= -(1-\delta_{k0})\sum_{j=0}^{k-1}
  \tilde{c}^k_{j,\varphi} \frac{u_3}{\varphi} Z_3 Z^{\tilde{\alpha}}_\hh Z^j_3 u  
  -\sum_{1\le |\beta|\le |\alpha|}
  {\alpha \choose \beta}
  Z^\beta u \cdot Z^{\alpha-\beta} \nabla u
  \\&
  =J_1 + J_2
  \end{split}
  .\label{EQ116}     
 \end{align}
The term $J_1$ is estimated as
 \begin{align}
  \Vert J_1\Vert_{L^\infty} \lec \Vert u\Vert_{1,\infty}\Vert u\Vert_{2,\infty}
  ,\label{EQ117} 
 \end{align}
while the term $J_2$ is treated as
 \begin{align}
  J_2 =
  \sum_{1\le |\tilde{\beta}| \le |\alpha|}
  {\alpha \choose \beta} 
  (Z^{\beta} u_\hh 
  \cdot 
  \nabla_\hh Z^{\alpha-\beta} u
  + 
  Z^{\beta} u_3 
  Z^{\alpha-\beta} \partial_z u)
  = J_{21} + J_{22}
  \lec \Vert u\Vert_{1,\infty}\Vert u\Vert_{2,\infty} + J_{22}.
  \label{EQ118}
 \end{align}
To bound $J_{22}$, we 
commute $\partial_z$ and $Z_3$ when necessary and write
  \begin{equation}
   \label{EQ119}
   \Vert Z^{\beta} u_3  Z^{\alpha - \beta} \partial_z u\Vert_{L^\infty}
   \lec
   \begin{cases}
     \left\Vert Z^{\beta} \frac{u_3}{\varphi}\right\Vert_{\infty} \Vert Z^{\alpha - \beta} Z_3 u\Vert_{L^\infty}
     \lec \Vert u\Vert_{2,\infty}^2, &|\beta|=1
      \\
     \Vert Z^{\beta} u_3\Vert_{L^\infty}\Vert Z^{\alpha - \beta} \partial_z u\Vert_{L^\infty}
     \lec \Vert u\Vert_{2,\infty}(\Vert \omega\Vert_{L^\infty}+\Vert u\Vert_{1,\infty})
     , &|\beta|=2
     .
   \end{cases}
 \end{equation}
  Combining \eqref{EQ116}--\eqref{EQ119}
 and integrating in time implies
 \begin{align}
  \int_0^t \Vert u\cdot \nabla Z^\alpha u - Z^\alpha(u\cdot \nabla u)\Vert_{L^\infty} 
  \Vert u\Vert_{2,\infty}\,ds
  \lec \int_0^t \Vert u\Vert_{2,\infty}^2(\Vert \omega\Vert_{L^\infty}+\Vert u\Vert_{2,\infty})\,ds
  .\label{EQ120}
 \end{align}
 Finally, we consider the pressure term on the right-hand side of \eqref{EQ114}.
 Employing \eqref{EQ.emb} yields
 \begin{align}
  \Vert \nabla p\Vert_{2,\infty}
  \lec
  \Vert \partial_z p\Vert_{2,\infty}+\Vert p\Vert_{3,\infty}
  \lec
  \Vert D^2 p\Vert_{3}^\frac{1}{2}\Vert \nabla p\Vert_{4}^\frac{1}{2}+\Vert D^2 p\Vert_{3}
  \lec
  \Vert D^2 p\Vert_{3}
  .\label{EQ121}
 \end{align}
 
 Now, we multiply \eqref{EQ101} by $\Vert \omega\Vert_{L^\infty}$ and integrate on $[0,T]$.
 Next, we sum the resulting inequality with \eqref{EQ114} and use \eqref{EQ120},~\eqref{EQ121}.
 It follows that
 \begin{align}
   \Vert u(t)\Vert_{2,\infty}^2+\Vert \omega(t)\Vert_{L^\infty}^2
   \lec 
   \Vert u_0\Vert_{2,\infty}^2+\Vert \omega_0\Vert_{L^\infty}^2
   + \int_0^t (\Vert u\Vert_{2,\infty}+\Vert \omega \Vert_{L^\infty}+1)^3
   +\Vert u\Vert_{2,\infty}\Vert D^2 p\Vert_{3}\,ds
   .\llabel{EQ302}
 \end{align}
 Using  $\Vert \omega\Vert_{L^\infty} \le \Vert \nabla u\Vert_{L^\infty}$ and \eqref{EQ300},
 we then obtain~\eqref{EQ.Inf1}.
\end{proof}

\subsection{Conclusion of the a~priori estimates}

Given $u_0$ as in Proposition~\ref{P.Ap}, denote by $u\in C([0,T];H^5(\Omega))$
a solution to~\eqref{NSE0}--\eqref{hnavierbdry} on $[0,T]$ for $T>0$.
Then, collecting \eqref{EQ.Con}, \eqref{EQ.Inf1}, and \eqref{EQ.Pre} yields
\begin{align}
  N^2(t)=
  (\Vert u(t)\Vert_{4}
  +\Vert u(t)\Vert_{2,\infty}
  +\Vert \nabla u(t)\Vert_{L^\infty})^2   
  \le C \left(N^2(0) + \int_0^t N^3(s)\,ds\right),
  \llabel{ap3}
\end{align}
for $t \in [0,T]$.
By Gr\"onwall's inequality, there exist $T_0 > 0$ and $M_0>0$ depending only on 
the norms of the initial data such that \eqref{ap} holds, concluding the proof of Proposition~\ref{P.Ap}.

\startnewsection{Proof of Theorem~\ref{T04}}{sec.main}

Let $u_0$ be as in Theorem~\ref{T04} and $\{u_0^r\}_{r>0} \in C^\infty(\Omega)$ a sequence of divergence-free 
smooth functions that are tangential on the boundary. In particular, $u_0^r \in H^5(\Omega)$ for all $r$, and
\begin{align}
  \begin{split}
    u_0^r &\to u_0 \text{ strongly in } H^4_\cco(\Omega),
    \\
    u_0^r &\rightharpoonup u_0 \text{ weakly-* in } W^{1,\infty}(\Omega)\cap W^{2,\infty}_\cco(\Omega),
    \llabel{EQ306}  
  \end{split}
\end{align}
as $r \to 0$.
Now, for a fixed $r>0$, there exists a unique solution $u^r \in C([0,T^r_{\text{max}});H^5(\Omega))$
of \eqref{NSE0}--\eqref{hnavierbdry} emanating from $u_0^r$, where $T^r_{\text{max}}$
denotes the maximal time of existence.
Then, recalling that $T_0$ is as in Proposition~\ref{P.Ap}, we have $T_0 \le T^r_{\text{max}}$.
Indeed, by the a~priori estimates
\eqref{ap} we have a uniform control of the Lipschitz norm of $u$ on the time interval
$[0,T_0]$. Therefore, $u^r \in C([0,T_0];H^5(\Omega))$, as a solution of the Euler equations,
 can be continued in the same class implying $T_0 \le T^r_{\text{max}}$.
 It follows that the sequence of approximate solutions 
 $u^r \in L^\infty(0,T_0;H^4_\cco \cap W^{1,\infty}\cap W^{2,\infty}_\cco)$
 are bounded uniformly in~$r$.
 
 Before passing to the limit, we now show that the sequence of 
 approximate solutions is Cauchy in $L^\infty(0,T_0;L^2(\Omega))$.
 To establish this, let $r_1, r_2 \in (0,1)$, and denote by $(u^1,p^1)$
 and $(u^2,p^2)$ two solutions to \eqref{NSE0}--\eqref{hnavierbdry}
 emanating from $u_0^{r_1}$ and $u_0^{r_2}$, respectively.
 Then, the difference of solutions 
 $(U,P)= (u_1-u_2,p_1-p_2)$ satisfies
 \begin{align}
  U_t 
  +U \cdot \nabla u^1 + u^2 \cdot \nabla U + \nabla P 
  =0, \text{ and } \div U =0,
  \llabel{diff}
 \end{align} 
 with the boundary conditions
 \begin{align}
  U_3 = 0 
  \comma (x,t) \in \{z=0\}\times (0,T_0)
  .\llabel{b.diff}
 \end{align}
 The usual $L^2$ estimates imply
 \begin{align}
   \frac{1}{2}\frac{d}{dt}\Vert U\Vert_{L^2}^2
   =
   - \int U \cdot \nabla u^1 U \,dx
   ,
   \llabel{EQ149}
 \end{align}
and recalling \eqref{ap}, we obtain
 \begin{align}
  \frac{d}{dt}\Vert U\Vert_{L^2}^2
  \lec
  \Vert U\Vert_{L^2}^2\Vert \nabla u^1\Vert_{L^\infty}
  \lec
  \Vert U\Vert_{L^2}^2
  ,\llabel{EQ151}
 \end{align}
 allowing the implicit constant to depend on~$M_0$.
 Therefore, employing the Gr\"onwall's inequality on $(0,T_0)$,
 it follows that 
 \begin{align}
  \sup_{[0,T]}\Vert U\Vert_{L^2}^2 \lec \Vert u_0^{r_1}-u_0^{r_2}\Vert_{L^2}^2
  ,\llabel{EQ154}
 \end{align}
 showing that $u^r \in L^\infty(0,T_0;L^2(\Omega))$ is a Cauchy sequence.
 Upon passing to a subsequence, we may pass to the limit in \eqref{NSE0}, concluding that there exists a solution 
 $u \in L^\infty(0,T_0;H^4_\cco\cap W^{1,\infty} \cap W^{2,\infty}_\cco)$ for \eqref{NSE0} such that
\begin{align}
  \begin{split}
    u^r &\to u \text{ strongly in } L^\infty(0,T_0;L^2(\Omega)),
    \\
    u^r &\rightharpoonup u \text{ weakly-* in } L^\infty(0,T_0;(H^4_\cco(\Omega)\cap W^{1,\infty}(\Omega)\cap W^{2,\infty}_\cco(\Omega))).
    \llabel{EQ307}  
  \end{split}
\end{align}
 Due to the Lipschitz regularity, $u$ is a unique solution, and it is continuous-in-time
 recalling that $C([0,T_0];H^3_\cco) \subseteq H^1(0,T_0;H^3_\cco)$. Indeed, applying 
 three conormal derivatives to \eqref{NSE0}, we may conormalize 
 the normal derivatives in $Z^3 (u\cdot \nabla u)$ and
 use \eqref{ap} and \eqref{EQ.Pre} to conclude that $u_t \in L^2(0,T_0;H^3_\cco)$.
 Finally, we pass to the limit in \eqref{hnavierbdry} using that
 \begin{align}
   \Vert u-u^r\Vert_{L^2(\partial \Omega)}\lec
    \Vert \nabla (u-u^r)\Vert_{L^\infty}^\frac{1}{5}\Vert u-u^r\Vert_{L^2}^\frac{4}{5}
     +\Vert u-u^r\Vert_{L^2}
     ,\llabel{EQ308}
 \end{align}
concluding the proof of Theorem~\ref{T04}.

\colb
\section*{Data availability statement}
The paper has no associated data.

\colb
\section*{Acknowledgments}
The authors were supported in part by the
NSF grant DMS-2205493.

\end{document}